\documentclass[11pt]{article}
\usepackage{amsmath}
\usepackage{amsthm}
\usepackage{amssymb}

\theoremstyle{plain}
\newtheorem{thm}{\protect\theoremname}
\theoremstyle{plain}
\newtheorem{lem}[thm]{\protect\lemmaname}
\theoremstyle{plain}

\newtheorem{conj}[thm]{\protect\conjecturename}
\theoremstyle{remark}

\theoremstyle{remark}
\newtheorem{rmk}[thm]{\protect\remarkname}
\theoremstyle{plain}

\theoremstyle{plain}

\theoremstyle{plain}

\usepackage{amssymb}

\makeatother
\providecommand{\conjecturename}{Conjecture}
\providecommand{\claimname}{Claim}
\providecommand{\corollaryname}{Corollary}
\providecommand{\lemmaname}{Lemma}
\providecommand{\definitionname}{Definition}
\providecommand{\remarkname}{Remark}
\providecommand{\theoremname}{Theorem}
\providecommand{\problemname}{Problem}
\providecommand{\propositionname}{Proposition}

\usepackage{pifont}

\newcommand{\oC}{\overline{C}}

\newcommand{\FF}{{\cal F}}

\def\lt{\left}
\def\rt{\right}

\def\<{\langle}
\def\>{\rangle}
\def\wh{\widehat}

\def\lt{\left}
\def\rt{\right}

\def\<{\langle}
\def\>{\rangle}

\newcommand{\PP}{{\mathbb P}}
\newcommand{\EE}{{\mathbb E}}

\newcommand{\E}{\mathbb{E}}

\begin{document}
\title{Random Cayley graphs and random sumsets}
\author{Noga Alon\thanks{Department of Mathematics, Princeton University, 
Princeton, NJ 08540. Email: {\tt nalon@math.princeton.edu}. 
Research supported by NSF grant DMS-2154082 and by the Clay
Institute.} 
\and Huy Tuan Pham\thanks{Department of Mathematics, California 
Institute of Technology, Pasadena, CA 91106. Email: 
{\tt htpham@caltech.edu}. Research supported by a 
Clay Research Fellowship.}}
\date{}
\maketitle
\begin{abstract}
We prove that any finite abelian group $G$ contains a collection $\FF$ of
not too many subsets with a special structure, so that for every
subset $A$ of  $G$ with a small doubling, there is a member
$F \in \FF$ that is fully contained in the sumset $A+A$ and is not
much smaller than it. Using this result we obtain improved bounds
for the problem of estimating the typical independence number of
sparse random Cayley or Cayley-sum graphs, 
and for the problem of estimating the
smallest size of a subset of $G$ which is not a sumset.
We also obtain tight bounds for the typical maximum length of an
arithmetic progression in the sumset of a sparse random subset of $G$.
\end{abstract}

\section{Introduction}

Given an abelian group $G$ and a finite subset $A$ of $G$, 
define the sumset  of $A$
\[
    A + A = \{a+b: a,b\in A\},
\]
and the associated doubling constant $K = \frac{|A+A|}{|A|}$. 
Sumsets and sets with small doubling $K$ are of fundamental interest 
in additive combinatorics. Over the years, multiple aspects of the 
structure of sumsets and sets with small doubling have been studied. 
A notable result is the influential Freiman-Ruzsa theorem 
\cite{Frei1, Frei2, R2}, generalized by Ruzsa \cite{R} 
and Green and Ruzsa \cite{GR}, which shows that sets $A$ of constant 
doubling $K$ must be dense subsets of certain structured objects 
known as coset progressions. 

While structural results typically provide information about the 
set $A$ given its doubling $K$, they provide relatively weak 
information about the sumset $A+A$. Motivated by fundamental 
applications to the investigation of sparse random Cayley graphs 
\cite{A07, A13, G05, GM16}, we study in this paper 
new perspectives on the structure of sumsets $A+A$ of sets 
with small doubling. 
Our main result provides an answer to the following basic question: 
Does every sumset $A+A$ of a set $A$ with small doubling $K$ 
fully contain a \emph{dense} structured subset $F$? 
In exploring the inherent structure of the sumset $A+A$, we 
also study natural questions about sumsets $A+A$ of binomial 
random sets $A$. Before describing the precise results, we 
first describe the motivating problems for our work. 
\vspace{5pt}

{\noindent \bf Sparse random Cayley graphs.} Given an abelian 
group $G$ and a symmetric subset of it $S$, 
the Cayley graph $\Gamma(G;S)$ of $G$ with generating set $S$ 
has vertex set $G$, and two group elements $x$ and $y$ are 
connected if and only if $y-x \in S$. A random Cayley graph 
$G(p)$ is obtained by selecting each equivalence class $\{x,-x\}$ 
to be in the generating set $S$ independently at random 
with probability $p$. 

We also consider here the Cayley sum graph $\Gamma^+(G;S)$ which, 
given a generating set $S$ (not necessarily symmetric), 
has vertex set $G$ and two group elements $x$ and $y$ are connected 
if and only if $x+y\in S$. A random Cayley sum graph $G^+(p)$ 
is obtained by selecting each element $x$ to be in $S$ 
independently at random with probability $p$. 

The independence number $\alpha(G(p))$ of random Cayley graphs 
has been extensively studied. In the dense case $p=1/2$ 
(and more generally $p=\Theta(1)$), the first author 
and Orlitsky \cite{AO95} showed that $\alpha(G(1/2)) = O((\log |G|)^2)$ 
with high probability. Here and in what follows we say that an
event holds with high probability (whp, for short) if the 
probability it holds tends to $1$ as the relevant parameter tends
to infinity.
Green \cite{G05} showed that whp $\alpha(G(1/2)) = \Theta(\log |G|)$ 
for cyclic groups $G = \mathbb{Z}_n$ and that 
$\alpha(G(1/2))=\Theta(\log |G|\log \log |G|)$ for finite field 
vector spaces $G = \mathbb{F}_2^d$ (a similar result holds for any finite field vector space $G = \mathbb{F}_p^d$ with $p$ fixed). Green and Morris \cite{GM16} later 
sharpened Green's result to show that whp
$\alpha(G(1/2)) = (2+o(1))\log_2 |G|$ for $G = \mathbb{Z}_n$. 
The problem is significantly harder in the sparse 
case $p = o_{|G|}(1)$. The best general result in this direction 
is the following theorem of the first author \cite{A07, A13}. 

\begin{thm}
	\label{t01}
    Let $G$ be a group of size $n$. The independence number of 
the random Cayley graph $G(p)$ is at 
	most $O(\min(p^{-2}(\log n)^2, \sqrt{n(\log n)/p}))$ 
whp. 
\end{thm}

It is natural to conjecture that, in terms of the independence number, 
random Cayley graphs behave similarly to random regular graphs 
of the same degree. Here and throughout the paper, we often 
hide polylogarithmic factors in $|G|$ using the
$\tilde{O}$ and $\tilde{\Omega}$ notation. 
\begin{conj}[\cite{A13}]
\label{conj:ind}
Let $G$ be a group of size $n$. The independence number of the 
random Cayley graph $G(p)$ is at most $\tilde{O}(p^{-1})$ whp. 
\end{conj}
Unlike random regular graphs, random Cayley graphs, 
and especially sparse random Cayley graphs, are significantly 
harder to analyze due to the limited randomness in their 
definition which causes significant dependencies. 

In a recent work \cite{CFPY}, motivated by Ramsey-theoretic applications, 
Conlon, Fox, the second author and Yepremyan gave an improvement of 
Theorem \ref{t01} for general groups $G$. 
\begin{thm}
    Let $G$ be a group of size $n$. The independence number of the 
random Cayley graph $G(p)$ is at most 
	$O\left(p^{-2} \log n \max\left(\log p^{-1}, p^{-1} 
	\log \left(\frac{\log n}{\log p^{-1}}\right)\right)\right)$ 
	whp. 
\end{thm}
This result is generally tight for the dense case $p=\Theta(1)$. 
On the other hand, for sparse $p$, such as $p = n^{-c}$ for 
some $c>0$, the result only gives lower order improvements. 
No improvement over the exponent $p^{-2}$ in Theorem \ref{t01} 
has been obtained so far. 

As one application of our key result, we obtain the 
first improvement in the exponent of $p$. 
\begin{thm}
\label{thm:ind-num}
%There exists a constant $C>0$ such that the following holds. 
Let $G$ be an abelian group of size $n$ and let
$p \leq 1/2$. Then the
independence number of the random Cayley graph $G(p)$ 
and the random Cayley sum graph $G^+(p)$ is at
most $\tilde{O}(p^{-3/2})$ whp. 
\end{thm}

{\noindent \bf Largest non-sumsets in $\mathbb{Z}_n$.} 
One of the applications of the independence number of polynomially 
sparse random Cayley sum graphs is toward the following question 
in additive combinatorics first considered by Green \cite{Gthe}. 
Let $f(n)$ be the largest integer such that every subset 
of $\mathbb{Z}_n$ with size larger than $n-f(n)$ can be 
represented as a sumset $A+A$. Green asked to determine or estimate $f(n)$ 
and showed that $f(n) \ge \Omega(\log n)$. 

The first author proved that $f(n) \ge \tilde{\Omega}(n^{1/2})$ 
and, via the upper bound $\alpha(G^+(p)) = O(p^{-2}(\log n)^2)$, 
that $f(n) \le \tilde{O}(n^{2/3})$. Using 
Theorem \ref{thm:ind-num}, we obtain an improvement in the exponent 
of the upper bound to $f(n)$.  

\begin{thm}
\label{thm:sumset}
In the notation above, $f(n) \le \tilde{O}(n^{3/5})$. 
Specifically, let $G$ be an abelian group of size $n$. Then there
exists a subset of $G$ of size at least $n-\tilde{O}(n^{3/5})$ 
which cannot be represented as a sumset $A+A$ for $A\subset G$. 
\end{thm}

{\noindent \bf The key result. }
Our key input to Theorem \ref{thm:ind-num} is a structural 
result showing that every sumset $A+A$ of a set $A$ with small 
doubling must fully contain a \emph{dense structured} 
subset $F$. 
\begin{thm}
\label{thm:eff-cover}
    There exist $C,c>0$ such that the following holds. 
Let $G$ be an abelian group of order $n$ and let $s\le n$. 
There exist collections $\cal{F}_\ell$ of subsets 
	of $G$ such that 
    \[
      |\mathcal{F}_\ell| \le \exp\left(C\min\left(2^{2\ell} 
	(\log n)^2, \sqrt{2^{\ell}s (\log n)^{3/2}}\right)\right),
    \] 
    and
    \[
        \min_{F\in \cal{F}_\ell} |F| \ge c2^{\ell} s /\ell^2,
    \]
    so that the following property holds.
    
    Let $A\subseteq G$ be such that $|A+A| \le K|A|$ and $|A| = s$. 
    Then there exists $\ell \le \log_2 K$ and $F\in \cal{F}_\ell$ 
	such that 
    \[
        A+A \supseteq F.
    \] 
\end{thm}
Informally speaking, the theorem yields that for every $s \leq n$
and every $K$, there exists a dyadic 
scale $h = 2^{\ell} \le K$ and a collection of sets $\mathcal{F}$ 
with complexity $\log |\mathcal{F}| \le \tilde{O}(\min(h^2,\sqrt{hs}))$,
such that for every subset $A$ of size $s$ and doubling at most $K$,
$A+A$ fully contains a set from $\mathcal{F}$ of size 
at least $\tilde{\Omega}(h|A|)$. 
A particularly surprising aspect of this result is the case 
the doubling $K$ is bounded. In this case we obtain that $A+A$ 
contains a structured subset $F$ which is dense in $A+A$, where $F$ 
lies in a collection of sets of complexity $\tilde{O}(1)$ (which, 
crucially, is independent of $A$).

Indeed, Lovett \cite{L} asked if it is possible to find a 
small collection $\mathcal{F}$ of dense sets such that, for 
every dense subset $A\subseteq G$, the sumset $A+A$ 
contains a member of $\mathcal{F}$. 
For every dense $A$, the doubling of $A$ is clearly bounded. 
As a special case of Theorem \ref{thm:eff-cover}, we thus 
resolve Lovett's question.

\begin{thm}
	\label{thm:struc}
   Let $G$ be an abelian group of order $n$. For any $\delta>0$, 
there are $\epsilon > 0$ and $C>0$ such that the following holds. 
There exists a collection $\mathcal{F}_\delta \subseteq 2^{G}$ 
consisting of sets of size at least $\epsilon n$ 
with $|\mathcal{F}_\delta| \le \exp(C(\log n)^2)$ 
so that for every $|A| \ge \delta n$, $A+A$ fully contains 
a set $F\in \mathcal{F}_\delta$. 
\end{thm}
It is easy to see that any collection $\mathcal{F}_\delta$ 
satisfying the property in Theorem \ref{thm:struc} must have 
size at least $\exp(\omega_\delta(1)(\log n))$. For example,
if $G = \mathbb{F}_2^d$, consider subsets $A$ given by 
codimension $\log_2(1/\delta)$ subspaces. Each set $F$ of size 
$|F| \ge \epsilon n$ can be a subset of at most 
$2^{\log_2(1/\epsilon)\log_2(1/\delta)}$ 
subspaces of codimension $\log_2(1/\delta)$. 
On the other hand, the number of subspaces of codimension 
$\log_2(1/\delta)$ is at least 
$2^{d\log_2(1/\delta) - (\log_2(1/\delta))^2}$. 
Hence we get \[|\mathcal{F}_\delta| 
\ge 2^{d\log_2(1/\delta) - \log_2(1/\epsilon)\log_2(1/\delta) 
- \log_2(1/\delta)^2}.\] 

\iffalse
$2^{\log_2(1/\epsilon)(\log_2(1/\epsilon)-\log_2(1/\delta))}$ 
subspaces of codimension $\log_2(1/\delta)$. 
On the other hand, the number of subspaces of codimension 
$\log_2(1/\delta)$ is at least 
$2^{d\log_2(1/\delta) - (\log_2(1/\delta))^2}$. 
Hence we need \[|\mathcal{F}_\delta| 
\ge 2^{d\log_2(1/\delta) - \log_2(1/\epsilon)^2 - \log_2(1/\delta)^2}.\] 
\fi

Theorem \ref{thm:eff-cover} connects directly with sparse random 
Cayley graphs via providing an efficient 
\emph{union bound obstruction} or \emph{small cover} to the existence 
of large independent sets in random Cayley graphs. 
In probabilistic combinatorics language, a cover for a 
collection $\mathcal{H}$ of sets is a collection $\mathcal{G}$ 
such that every set in $\mathcal{H}$ contains a set in 
$\mathcal{G}$ as a subset. In our case the collection 
$\mathcal{H}$ is the collection of sumsets $A+A$ 
over sets $A$ of suitable size. In this language, 
the collection $\bigcup_{\ell} \mathcal{F}_\ell$ describes a 
cover for the collection $\mathcal{H}$. By the union bound 
over $\bigcup_\ell \mathcal{F}_{\ell}$, the random Cayley sum 
graph $G^+(p)$ typically does not contain an independent set of 
size $s$ if this cover is small, that is, if
\[
	\sum_{\ell} \sum_{F\in \mathcal{F}_\ell} (1-p)^{|F|} =o(1).
\]
Union bound obstructions play a crucial role in the study of thresholds. 
In particular, the Kahn-Kalai conjecture \cite{KK}, proved in
\cite{PP-K}, implies direct connections between thresholds and 
union bound obstructions. We expect that the independence number 
of random Cayley graphs can be accurately determined via an 
optimally efficient cover for the collection $\mathcal{H}$ 
of sumsets (see Conjecture \ref{conj:cover} in the final section). 
\vspace{5pt}

{\noindent \bf Arithmetic progressions in random sumsets. }
Our results here suggest that, even for arbitrary sets $A$, 
the sumset $A+A$ contains large nontrivial structures of low complexity. 
In the investigation of the structure of the collection of 
sumsets $A+A$, it is natural to study the behavior of typical 
sumsets, that is, the sumset $A+A$ of a random set $A$. 
In the cyclic group $\mathbb{Z}_p$, we consider a random set $A$ 
obtained by including each element of $\mathbb{Z}_p$ independently 
with probability $q$. Here we study specifically the length of the 
longest arithmetic progression contained in $A+A$. 

For arbitrary dense $A$, the length of the longest arithmetic 
progression in $A+A$ has been extensively studied \cite{G02, CLS}. 
In a recent paper of Kohayakawa and Miyazaki \cite{KM} the authors
consider the problem of estimating the typical maximum length 
$ap(A+A)$ of an
arithmetic progression of the sumset $A+A$ of 
a random subset $A$ of $[n]=\{1,2, \ldots ,n\}$. 

The authors of \cite{KM} show that if every element is chosen to
lie in $A$, randomly
and independently, with probability $q=q(n)$, then $ap(A+A)$ 
exhibits a sharp change of behavior around $q=\frac{1}{\sqrt
n}$.
In particular, they prove that for 
$$
q=\frac{1}{\sqrt {n (\log n)^{\Theta(1)}}}
$$ 
the maximum length is 
$\Theta( \frac{\log n}{\log \log n})$ whp,
and that for 
$q \geq \sqrt{\frac{4\log n}{n}}$ the maximum is
$\Theta(n)$ whp. For $q=\frac{1}{\sqrt n}$ their arguments only
suffice to establish an $\Omega( \log n/ \log \log n)$ lower bound
and an $O(n)$ upper bound.

Here we prove that for this probability the correct
answer is $\Theta(\log n)$ whp. It is more convenient to study the 
closely related analogous problem for the finite cyclic group $Z_p$
for prime $p$. The
same arguments provide a similar result for a random subset of
$[n]$. 

Let $p$ be a large prime, and let $A$
be a random subset of $Z_p$, where each $a \in Z_p$ lies
in $A$, randomly and independently, with probability
$q=1/\sqrt p$.  
Let $ap(A+A)$ denote, as above, the
maximum length of an arithmetic progression in $A+A$. 

\begin{thm}
\label{t11}
There exist two absolute positive constants $c_1,c_2$ so that for
$A$ as above, the maximum length
$ap(A+A)$  of an arithmetic progression in $A+A$  satisfies, whp,
$$ 
c_1 \log p \leq ap(A+A) \leq c_2 \log p.
$$
\end{thm}

{\noindent \bf Paper organization. }
In Section \ref{sec:cover} we prove our key result, 
Theorem \ref{thm:eff-cover} and Theorem \ref{thm:struc}. 
In Section \ref{sec:app-cover} we prove Theorems 
\ref{thm:ind-num} and \ref{thm:sumset} as applications of the 
main covering lemma Theorem \ref{thm:eff-cover}. 
In Section \ref{sec:ap} we study random sumsets and prove 
Theorem \ref{t11}. The final Section \ref{s5} 
contains some concluding remarks and
open problems.

\section{The main covering lemma}\label{sec:cover}

\subsection{Preliminaries on discrete Fourier analysis}
Let $G$ be an abelian group of order $n$. Denote by $\wh{G}$ 
the group of characters $\chi:G \to \mathbb{C}$. 
Given $f:G\to \mathbb{C}$, the Fourier transform of $f$ is defined as 
\[
    \wh{f}(\chi) = \E[f(x)\chi(x)].
\]
For the convolution 
\[
    f*g (x) = \E_y[f(y)g(x-y)],
\]
we have that 
\[ 
    \wh{f*g}(\chi)=\wh{f}(\chi)\wh{g}(\chi).
\]
Parseval's identity asserts that 
\[
    \sum_\chi |\wh{f}(\chi)|^2 = \E [ |f(x)|^2 ].
\]
Finally, we have the Fourier inversion formula
\[
    f(x) = \sum_{\chi \in \widehat{G}} \wh{f}(\chi)\chi(-x).
\]

\subsection{The covering theorem}
In this section we prove the main covering results, 
Theorem \ref{thm:eff-cover} and Theorem \ref{thm:struc}. Throughout
the section, $G$ is an abelian group of size $n$ and $A$ is 
a subset of $G$.
Identifying $A$ with its characteristic function, 
observe that $A+A = \mathrm{supp}(A*A)$ and 
$A-A = \mathrm{supp}(A*(-A))$. Our main tool is the following 
result giving Fourier sparse pointwise approximation to the 
convolutions $A*A$ and $A*(-A)$. 

\begin{lem}\label{lem:conv}
    Let $\alpha \in (0,1)$, $\eta\in (0,1)$ and 
	$m = 4\eta^{-2}(\log n)$. Then for all $A\subseteq G$ 
	of size $\alpha n$, there exist $\chi_1,\dots,\chi_m$ such that 
    \begin{equation}\label{eq:Foudiff}
    \max_x \lt|A*(-A)(x) - \frac{\alpha}{m}
	    \sum_{i=1}^{m}\chi_i(-x)\rt| \le \eta \alpha. 
    \end{equation}
    Similarly, there exist $\chi_1,\dots,\chi_m$ such that 
    \begin{equation}\label{eq:Fousum}
    \max_x \lt|A*A(x) - \frac{\alpha}{m}\sum_{i=1}^{m}\frac{\wh{A}
	   (\chi_i)^2}{|\wh{A}(\chi_i)|^2}\chi_i(-x)\rt| \le \eta \alpha. 
    \end{equation}
\end{lem}
\begin{proof}
    We have \[A*(-A)(x) = \sum_{\chi} |\widehat{A}(\chi)|^2 \chi(-x).\]
    Note that $\sum_\chi |\wh{A}(\chi)|^2 = \alpha$ by Parseval's identity.
    For each $i \in [m]$, pick $\chi_i$ independently such that \[
    \PP(\chi_i = \chi) = \alpha^{-1}|\widehat{A}(\chi)|^2.\] Then 
    \[
    \mathbb{E} \left[\frac{\alpha}{m}\sum_{i=1}^{m}\chi_i(-x)\right] 
	= A*(-A)(-x). 
    \]
    Furthermore, noting that $|\chi(-x)|\le 1$, 
    by the Hoeffding bound:
    \[
        \PP\lt[\lt|\frac{\alpha}{m}\sum_{i=1}^{m}\chi_i(-x) 
	- A*(-A)(x)\rt| > t\frac{\alpha}{\sqrt{m}}\rt] \le 2\exp(-t^2/2).
    \]
    By the union bound, we then have that with positive probability, 
    \[
        \max_x \lt|A*(-A)(x) - \frac{\alpha}{m}\sum_{i=1}^{m}\chi_i(-x)
\rt| \le 2\sqrt{\log n} \cdot \frac{\alpha}{\sqrt{m}} \le \eta\alpha. 
    \]

    For (\ref{eq:Fousum}), we follow the same proof, noting that 
    \[
        \EE\lt[\frac{\alpha}{m}\sum_{i=1}^{m} \frac{\wh{A}(\chi_i)^2}
{|\wh{A}(\chi_i)|^2} \chi_i(-x)\rt] = A*A(x). \qedhere
    \]
\end{proof}

We are now ready to prove the main covering theorem. Informally, 
the theorem asserts that the collection of difference sets 
$\{A-A: |A| = \alpha n, |A-A| \le K|A|\}$ admits a small \emph{cover}. 
For $A\subseteq G$ and an integer $\ell$, let 
    \[
        A_\ell = \{x: A*(-A)(x) \in (2^{-\ell-1}\alpha, 
	2^{-\ell} \alpha]\}.
    \]

We start with the following simple lemma.
\begin{lem}
	\label{lem:dyadic}
	There is an absolute constant $c>0$ such that the following
	holds.
    Let $G$ be an abelian group of order $n$. Let $\alpha \in (0,1)$. 
	Let $A\subseteq G$ be such that $|A-A| \le K|A|$ 
	and $|A| = \alpha n$. 
    Then there exists $\ell \le \log_2 K$ such that 
	$|A_\ell| \ge c2^{\ell}\alpha n/\ell^2$.
\end{lem}
\begin{proof}
    Note that $\EE_x A*(-A)(x) = \alpha^2$. Thus, if for 
	all $\ell \le \log_2 K$, 
    \[
        |A_\ell| < c 2^{\ell} \alpha n / \ell^2,
    \]
    then 
    \[
        \alpha^2 n < \sum_{\ell \le \log_2 K} 2^{-\ell} 
	\alpha \cdot c 2^{\ell} \alpha n / \ell^2 + 2^{-(\log_2 K+1)} 
	\alpha K|A| <\alpha^2 n,
    \]
    which is a contradiction for a suitable constant $c$.
\end{proof}

For $A\subseteq G$ such that $|A-A| \le K|A|$ and $|A| = \alpha n$, 
we denote by $\ell(A)$ the smallest index $\ell \le \log_2 K$ 
such that $|A_\ell| \ge c2^{\ell}\alpha n/\ell^2$.

\begin{thm}
\label{thm:cover-main}
There are absolute constants $C,c>0$ such that the following holds.
    Let $G$ be an abelian group of order $n$. Let $\alpha \in (0,1)$. 
There exist collections $\cal{F}_\ell$ of subsets of $G$ such that 
    \[
        |\mathcal{F}_\ell| \le \exp(C2^{2\ell} (\log n)^2),
    \] 
    and
    \[
        \min_{F\in \cal{F}_\ell} |F| \ge c2^{\ell}\alpha n/\ell^2,
    \]
    and the following property holds.
    
    Let $A\subseteq G$ be such that $|A-A| \le K|A|$ and $|A| = \alpha n$. 
    Then for $\ell = \ell(A)$ there is $F\in \cal{F}_\ell$ such that 
    \[
        A-A \supseteq F \supseteq A_\ell.
    \]
\end{thm}
\begin{proof}
    For each $\ell$, we define $\cal{G}_\ell$ as the collection 
	of functions of the form 
    \[
        \frac{\alpha}{m} \sum_{\chi \in S} \chi(-x),
    \]
    for $S\subseteq \wh{G}$ of size $2^{2\ell+6}(\log n)$. Let 
    $\cal{F}_\ell$ denote the collection of sets of the form 
    \[
        \{x: \wh{f}(x) > 2^{-\ell-2}\alpha\},
    \]
    for $\wh{f}\in \cal{G}_\ell$. 
    
    For each set $A$ with $|A-A| \le K|A|$, let $\ell = \ell(A)$ be 
	as in Lemma \ref{lem:dyadic}. For $\eta = 2^{-\ell-2}$, 
	we then find $S \subseteq \wh{G}$ of size $4\eta^{-2} (\log n)$ 
	such that 
    \[ 
        \lt | A*(-A)(x) - \frac{\alpha}{m} \sum_{\chi \in S} 
	\chi(-x)\rt | \le \eta \alpha.
    \]
    Let $\wh{f}_A(x) = \frac{\alpha}{m} \sum_{\chi \in S} \chi(-x)$ 
	and note that $\wh{f}_A\in \cal{G}_\ell$. Let
    \[
        F(A) = \{x: \wh{f}_A(x) > \eta \alpha\},
    \]
    so $F(A)\in \cal{F}_\ell$. We then have 
    \[
        F(A)\supseteq \{x: A*(-A)(x) > 2\eta \alpha\} \supseteq A_\ell,
    \]
    as desired.
\end{proof}

A similar proof yields an analogous result for sumsets $A+A$. 
For $A\subseteq G$ and an integer $\ell$, let 
    \[
        A^+_\ell = \{x: A*A(x) \in (2^{-\ell-1}\alpha, 2^{-\ell} \alpha].
    \]

\begin{thm}
	\label{thm:cover-sum}
	There are constants $C,c>0$ such that the following holds.
    Let $G$ be an abelian group of order $n$. Let $\alpha \in (0,1)$. 
	There exists collections $\cal{F}_\ell$ of subsets of $G$ 
	such that 
    \[
        |\mathcal{F}_\ell| \le \exp(C2^{2\ell} (\log n)^2),
    \] 
    and
    \[
        \min_{F\in \cal{F}_\ell} |F| \ge c2^{\ell}\alpha n/\ell^2,
    \]
    and the following property holds.
    
    Let $A\subseteq G$ be such that $|A+A| \le K|A|$ and $|A| = \alpha n$. 
    Then for $\ell = \ell(A)$ there is $F\in \cal{F}_\ell$ such that 
    \[
        A+A \supseteq F \supseteq A_\ell.
    \]
\end{thm}
\begin{proof}
    Let $\eta = c/n$ and let $\mathcal{U}$ be an $\eta/32$-net for 
	$\{x\in \mathbb{C}: |x|=1\}$ of size $O(\eta^{-1})$. 
	For each $\ell$, we define $\cal{G}_\ell$ as the collection 
	of functions of the form 
    \[
        \frac{\alpha}{m} \sum_{\chi \in S} t_\chi \chi(-x),
    \]
    for $S\subseteq \wh{G}$ of size $2^{2\ell+6}(\log n)$ 
	and $t_\chi \in U$ for each $\chi \in S$. Let 
    $\cal{F}_\ell$ denote the collection of sets of the form 
    \[
        \{x: \wh{f}(x) > 2^{-\ell-2}\alpha\},
    \]
    for $\wh{f}\in \cal{G}_\ell$. The rest of the proof is very 
	similar to the proof of Theorem \ref{thm:cover-main}. 
\end{proof}

We also observe that in the range of large doubling, we have the 
following version of the covering theorem. 
\begin{thm}
	\label{thm:cover-large}
	There are $C,c>0$ satisfying the following.
    Let $G$ be an abelian group of order $n$. Let $\alpha \in (0,1)$. 
	There exists collections $\cal{H}_\ell$ 
	of subsets of $G$ such that 
  \[
  |\mathcal{H}_\ell| \le \exp\lt(C\sqrt{2^{\ell}|A|(\log n)^{3/2}}\rt),
  \] 
    and
    \[
        \min_{H\in \cal{H}_\ell} |H| \ge c2^{\ell}\alpha n/\ell^2,
    \]
    and the following property holds.
    
    Let $A\subseteq G$ be such that $|A-A| \le K|A|$ and $|A| = \alpha n$. 
    Then for $\ell = \ell(A)$ there is $H\in \cal{H}_\ell$ such that 
    \[
        A-A \supseteq H \supseteq A_\ell.
    \]
\end{thm}
\begin{proof}
    We construct \[\mathcal{H}_{\ell} = \lt\{A' - A': |A'|\le 
	\sqrt{2^{\ell+4}|A|\log n}, |A'-A'| \ge c2^{\ell}
	\alpha n/\ell^2\rt\}.\]
    
    Let $|A|=\alpha n$ with $|A-A|\le K|A|$. As in the proof of 
	Theorem \ref{thm:cover-main}, there exists $\ell$ such that 
    \[
        |A_\ell| \ge c2^{\ell}\alpha n/\ell^2.
    \]
    
   Consider a random subset $A'$ of $A$ where each element is 
sampled independently with probability 
$q = (2^{\ell+3}(\log n)/|A|)^{1/2}$. Then we have, for $x\in A_\ell$, 
    \[
        \PP(x \notin A' - A') \le (1-q^2)^{2^{-\ell-1}|A|} \le n^{-2}. 
    \]
    Thus, combining with the standard Chernoff bound, with high 
	probability, $A'-A'\supseteq A_\ell$, and $|A'| \le 
	\sqrt{2^{\ell+4}|A|\log n}$. 

    In particular, there exists $H\in \cal{H}_\ell$ such that 
    \[
        A-A\supseteq H\supseteq A_\ell. \qedhere
    \]
\end{proof}

By the same proof, we also obtain the sumset analog of Theorem \ref{thm:cover-large}. 
\begin{thm}
	\label{thm:cover-large-sum}
	There are $C,c>0$ satisfying the following.
    Let $G$ be an abelian group of order $n$. Let $\alpha \in (0,1)$. 
	There exists collections $\cal{H}_\ell$ 
	of subsets of $G$ such that 
  \[
  |\mathcal{H}_\ell| \le \exp\lt(C\sqrt{2^{\ell}|A|(\log n)^{3/2}}\rt),
  \] 
    and
    \[
        \min_{H\in \cal{H}_\ell} |H| \ge c2^{\ell}\alpha n/\ell^2,
    \]
    and the following property holds.
    
    Let $A\subseteq G$ be such that $|A+A| \le K|A|$ and $|A| = \alpha n$. 
    Then for $\ell = \ell(A)$ there is $H\in \cal{H}_\ell$ such that 
    \[
        A+A \supseteq H \supseteq A_\ell.
    \]
\end{thm}

Theorems \ref{thm:cover-main}, \ref{thm:cover-sum}, \ref{thm:cover-large} and \ref{thm:cover-large-sum} imply Theorem \ref{thm:eff-cover}. 
Theorem \ref{thm:struc} is a direct corollary of 
Theorem \ref{thm:cover-sum}, noting that $K\le 1/\delta$ 
and $\ell(A) \le \log_2 K$. 

\section{Applications of the main covering lemma}\label{sec:app-cover}

\subsection{Independence number of random Cayley graphs}

As mentioned earlier, the efficient cover constructed for the 
collection of difference sets or sumsets in Theorem \ref{thm:eff-cover} 
immediately yields improved bounds for the independence number of 
random Cayley graphs or random Cayley sum graphs, 
Theorem \ref{thm:ind-num}. Here, we deduce Theorem \ref{thm:ind-num} 
based on Theorem \ref{thm:eff-cover} by an application 
of the union bound.

\begin{proof}[Proof of Theorem \ref{thm:ind-num}]
Consider $A\subseteq G$ with $|A| = s := \xi p^{-3/2} (\log n)^{19/4}$. 
Let $\ell = \ell(A)$. Let $\ell_0$ be so that $2^{\ell_0}$ 
is within a factor $2$ of $p^{-1/2}(\log n)^{3/4}$. 

For $\ell \ge \ell_0$, let $\cal{E}_\ell$ denote the event that 
the complement of the random generating set of the Cayley 
graph $S$ contains $H$ for some $H\in \cal{H}_\ell$. 
For $\ell < \ell_0$, let $\cal{E}_\ell$ denote the event that 
the complement of the random generating set of the Cayley graph $S$ 
contains $F$ for some $F\in \cal{F}_\ell$. Then the event that the 
independence number of $\Gamma(G; S)$ is at least 
	$s=\xi p^{-3/2}(\log n)^{19/4}$ is contained in the union 
of the events $\cal{E}_\ell$ for $\ell \ge \ell_0$ 
and $\cal{E}_\ell$ for $\ell < \ell_0$. 

By the union bound, we have
\begin{align*}
    &\PP[\alpha(\Gamma(G;S)) > s] \\
    &\le \sum_{\ell < \ell_0} |\mathcal{F}_\ell| 
(1-p)^{\min_{F\in \mathcal{F}_\ell}|F|} 
+ \sum_{\ell \ge \ell_0} |\mathcal{H}_\ell| 
	(1-p)^{\min_{H\in \mathcal{H}_\ell}|H|} \\
    &\le \sum_{\ell < \ell_0} \exp\lt(C2^{2\ell} (\log n)^2\rt) 
\exp\lt(-pc2^\ell s/\ell^2\rt) \\
&\qquad + \sum_{\ell \ge \ell_0} \exp\lt(C\sqrt{2^\ell (\log n)^{3/2} s}\rt)
	\exp\lt(-pc2^\ell s/\ell^2\rt).
\end{align*}
%$$
%    \PP[\alpha(\Gamma(G;S)) > s] 
%$$
%$$
%\le \sum_{\ell < \ell_0} |\mathcal{F}_\ell| 
%(1-p)^{\min_{F\in \mathcal{F}_\ell}|F|} 
%+ \sum_{\ell \ge \ell_0} |\mathcal{H}_\ell| 
%	(1-p)^{\min_{H\in \mathcal{H}_\ell}|H|} 
%$$
%$$
%\le \sum_{\ell < \ell_0} \exp\lt(C2^{2\ell} (\log n)^2\rt) 
%\exp\lt(-pc2^\ell s/\ell^2\rt) 
%$$
%$$
%+ \sum_{\ell \ge \ell_0} \exp\lt(C\sqrt{2^\ell (\log n)^{3/2} s}\rt)
%	\exp\lt(-pc2^\ell s/\ell^2\rt).
%$$
%$$
% \le \exp\lt(4C p^{-1}(\log n)^{7/2}\rt) \exp\lt(-c\xi p^{-1} 
%(\log n)^{7/2} / 2\rt).
%$$

\iffalse
\begin{align*}
    &\PP[\alpha(\Gamma(G;S)) > s] \\
    &\le \sum_{\ell < \ell_0} |\mathcal{F}_\ell| 
(1-p)^{\min_{F\in \mathcal{F}_\ell}|F|} + \sum_{\ell \ge \ell_0} 
|\mathcal{H}_\ell| (1-p)^{\min_{H\in \mathcal{H}_\ell}|H|} \\
    &\le \lt(\sum_{\ell < \ell_0} \exp\lt(C2^{2\ell} (\log n)^2\rt) 
+ \sum_{\ell \ge \ell_0} \exp\lt(C\sqrt{2^\ell (\log n)^{3/2} s}\rt)\rt)
\exp\lt(-pc2^\ell s/\ell^2\rt) \\
    &\le \exp\lt(4C p^{-1}(\log n)^{7/2}\rt) \exp\lt(-c\xi p^{-1} 
(\log n)^{7/2} / 2\rt).
\end{align*}
\fi

We next check that for
a sufficiently large absolute constant $\xi$ each of the two sums
above is small. To bound the first sum it suffices to check that
for every $\ell < \ell_0$ 
$$
	pc2^{\ell}s/\ell^2 \geq 2 \cdot C 2^{2\ell} (\log n)^2
	$$
	Since $\ell \leq \log n$ this is the case provided
	$$
	pcs/(\log n)^2  > 2 C 2^{\ell} (\log n)^2
	$$
	Substituting the values of 
	$s=\xi p^{-3/2} (\log n)^{19/4} $
it follows that this is equivalent to the inequality
$ c \xi p^{-1/2} (\log n)^{3/4} > 2C 2^{\ell}$
which clearly holds for each $\ell < \ell_0$, since
$2^{\ell_0}$ is within a factor of $2$ of $p^{-1/2} (\log n)^{3/4}$.

Note that the main contribution for this sum is given by 
$\ell=0$, implying that for an appropriate $\xi$ this sum is bounded
by $\exp(-\Omega(p^{-1/2}(\log n)^{11/4}))$.

In order to bound the second sum it suffices to check that
for $\ell \geq \ell_0$,
$$
2^{\ell}pcs/(\log n)^2 > 2C \sqrt {2^{\ell} (\log n)^{3/2} s}.
$$
This is equivalent to the inequality
$$
2^{\ell} \geq \frac{4C^2}{c^2 \xi} p^{-1/2} (\log n)^{3/4},
$$ 
which  clearly holds for all $\ell \geq \ell_0$ by the choice of 
$\ell_0$, provided $\xi$ is a sufficiently large constant. Here the main
contribution for the sum is obtained for $\ell=\ell_0$, showing that the
second sum is bounded by $\exp(-\Omega(p^{-1} (\log n)^{7/2})).$

Altogether, this computation shows that
\[
	\PP[\alpha(\Gamma(G;S)) > s] \le 
	\exp(-\Omega(p^{-1/2}(\log n)^{11/4})). \qedhere
\]
Via a similar argument, we obtain similar bounds for the upper 
bound to the independence number of random Cayley sum graphs. 
\end{proof}
It is worth noting that the polylogarithmic factor  
above can be improved, we make here no serious attempt
to optimize it.
\iffalse
\begin{rmk}
    The power of the log factor is not optimized, further 
	optimization reduces the bound 
	to $O(p^{-3/2}(\log n)^{13/4 + o(1)})$. 
\end{rmk}
\fi

\subsection{Large sets which are not sumsets}

Given Theorem \ref{thm:ind-num}, we give a short proof of 
Theorem \ref{thm:sumset}, showing that there exists a subset 
of $\mathbb{Z}_n$ of size $n-\tilde{O}(n^{3/5})$ that cannot 
be represented as a sumset. 

\begin{proof}[Proof of Theorem \ref{thm:sumset}]
    For a suitable value of $p$ which we will choose later, we let 
	$S_1$ be a random subset of $G$ where each element is chosen 
	with probability $p$. We will then let $A = G\setminus 
	(S_1 \cup S_2)$, and show that there exists a choice 
	of $S_2$ of size  $\Theta(pn)$ for which $A$ cannot be written 
	as a sumset $B+B$. Note that if $A = B+B$ then $B+B$ 
	is disjoint from $S_1$ and hence $B$ is an independent 
	set in $\Gamma^+(G; S_1)$. 

    By Theorem \ref{thm:ind-num} (and its proof above), 
with high probability, 
	$|S_1|\le 2pn$ and any independent set $B$ 
	in $\Gamma^+(G; S_1)$ has size at most 
	$\xi p^{-3/2}(\log n)^{19/4}$. Hence, given $S_1$, for any 
	choice of $S_2$, the number of potential choices for $B$ 
	such that $B+B = G\setminus (S_1 \cup S_2)$ is at most 
$$
	\sum_{i \leq \xi p^{-3/2}(\log n)^{19/4}}   \binom{|G|}{i}
	\le \exp\lt(\xi p^{-3/2}(\log n)^{23/4}\rt).
$$ 
	On the other hand, the number of choices for $S_2$ is at least 
	$\exp\lt(\Omega(pn\log n)\rt)$. 
	Hence, for $p = \xi' n^{-2/5} (\log n)^{19/10}$ for 
	a sufficiently large constant $\xi'$, there must exist a 
	choice of $S_2$ so that no $B$ satisfies $B+B = 
	G\setminus (S_1\cup S_2)$. 
\end{proof}

\section{Arithmetic progressions in sumsets of random sets}
\label{sec:ap}

In this section, we study $A+A$ for a random subset 
$A$ of $\mathbb{Z}_p$ where each element is included independently 
with probability $q = 1/\sqrt{p}$ and prove Theorem \ref{t11} 
regarding the length of the longest arithmetic progression 
contained in $A+A$. The proof is described in the following 
two subsections. The upper bound
is proved by applying Talagrand's Inequality (with a twist). The lower
bound is established by a second moment argument. 

\subsection{The upper bound}

In this subsection we prove the upper bound in Theorem 
\ref{t11}. Put
$k=c \log p$, where $c$ is a constant, to be chosen later.
\vspace{0.2cm}

\noindent
{\bf Claim 1:}\, 
With high probability, the set $A$ does not contain $5$ elements 
inside any progression of length $k$.
\vspace{0.2cm}

\noindent
We apply a simple union bound.
There are less than $p^2$ arithmetic
progressions $b_1,b_2, \ldots b_k$. 
In each progression there are less than
$k^5$ possibilities to 
choose $5$ elements, and the probability all of them belong to 
$A$ is $q^5=1/p^{2.5}$.  By the union bound the probability
that $A$ contains $5$ elements
inside a progression of length $k$ is thus at most
$p^2 k^5/p^{2.5}$ which is $O(\frac{\log^5 p}{p^{1/2}})=o(1)$, 
proving the claim

Note that if the assertion of Claim 1 holds then $A+A$ does not contain 
$5$ elements in a progression as above, with all of them
being the sum of a single element $a$
of $A$ with $5$ elements of $A$ (possibly including $a$ itself).
\vspace{0.2cm}

\noindent
{\bf Claim 2:}\,
For the right choice of the constant $c$, whp
there is no progression of $k=c \log p$ terms, so that each of 
its terms is
a sum of two elements of $A$, where no element of $A$ is used more 
than $4$ times in these $k$ sums.

Note that if the assertions of both claims hold, then $A$ contains
no $k$-term Arithmetic Progression.

To prove Claim 2 we fix a progression $P$ of $k$ terms
and use
Talagrand's Inequality to bound the probability of
getting all elements of $P$ in $A+A$ by sums of the required form
(that is, sums where no element of $A$ is used more than $4$ times).
We show that this probability is much smaller than $1/p^2$ and then
conclude that Claim 2 holds by the union bound.

Let $X$ be the random variable which is the maximum 
cardinality of a subset 
$T$ of $P$ that can be expressed by
sums of pairs of elements of $A$, with no element used more than
$4$ times.  Clearly this is a
$4$-Lipschitz function. It is not difficult to check that the 
expectation of $X$ is at most $0.4 k$ (even without the extra constraint
that no $a \in A$ is used more than $4$ times.)
Indeed, for each fixed element $g \in Z_p$, the probability that
$g$ is not in $A+A$ is precisely 
$$
(1-q)(1-q^2)^{(p-1)/2} =(1+o(1))1/\sqrt e>0.6
$$
Therefore the probability that $g \in A+A$ is smaller than $0.4$ and 
the desired estimate follows by linearity of expectation.

Another simple fact is that for any integer $t$, if
$X \geq t$, then this can be certified by 
at most $2t$ coordinates of the random vector of the $p$ random
choices determining which $g \in Z_p$ belong to $A$.

Therefore, Talagrand's inequality (see, for example, 
\cite{AS}, Section 7.7) implies that for every $b$ and $t$
\begin{equation}
\label{e21}
Prob[X \leq b -4t \sqrt{2b}] \cdot Prob[X \geq b] \leq e^{-t^2/4}.
\end{equation}
Taking in the inequality above $b$ as the median of $X$,  this implies that
the probability that $X$ is much smaller than the median is very small.
Similarly, taking $b$ and $t$ so that $b-4t \sqrt{2b}$ is the 
median shows that the probability that $X$ is much larger than the median
is also very small. This implies that the median is close to the 
expectation, thus it is at most $(0.4+o(1))k<k/2$.
We can now substitute $b=k$ and choose $t$ so that 
$k-4t \sqrt{2k}$ is the median. This gives $t=\Theta(\sqrt k)$.
Thus (\ref{e21}) provides an upper bound for the probability that 
$X=k$. 
This upper bound is of the form $e^{-\Theta(k)}=e^{-\Theta(c \log p)} $,
which is (much) less than $1/p^2$ for an appropriate
choice of $c$. This completes the proof of Claim 2 and 
hence that of the upper bound in Theorem \ref{t11}.

\subsection{The lower bound}

In the proof of the lower bound in Theorem \ref{t11} it will be convenient
to restrict attention only to arithmetic progressions of length $k$
in which each term is a sum of two distinct elements of $A$, all the 
relevant $2k$ elements of $A$ are pairwise distinct, and moreover all
the ${{2k} \choose 2}$ sums of pairs of these $2k$ elements are distinct. 

For every fixed arithmetic progression $P=(b_1,b_2, \ldots ,b_k)$ of
length $k$ in $Z_p$, where the difference is smaller than $p/2$,
and for every ordered sequence 
$C=(\{c_1,c_2\}, \ldots ,\{c_{2k-1},c_{2k}\})$ 
of $k$ unordered pairs
of elements of $Z_p$ satisfying 
$c_{2i-1}+c_{2i}=b_i$ for all $1 \leq i \leq k$, where
all the elements $c_i$ are distinct and all pairs of 
their sums are distinct,
let $X(P,C)$ be the indicator
random variable whose value is $1$ iff all elements
$c_j$ belong to the random set $A$.
Note that if this random variable is
$1$ then $A+A$ contains the arithmetic progression $P$.

It is not difficult to check that the number of these 
indicator random variables is 
$$
(1-o(1))\frac{p^2}{2} \left(\frac{p-1}{2}\right)^k
$$
Indeed, there are $p(p-1)/2$ progressions $P=(b_1,b_2, 
\ldots ,b_k)$ (obtained by choosing $b_1,b_2$).
For each $i$ there are $(p-1)/2$ unordered pairs of distinct
elements $c_{2i-1},c_{2i}$ whose sum is $b_i$.  The number of
choices in which not all the elements $c_j$ are distinct or
two sums of pairs of them are identical is negligible with respect
to the total number of choices.

Let $X$ be the sum of all these indicator 
random variables. Note that if $X>0$ then
$A+A$ contains an arithmetic progression of length $k$. 

For a set $C$ of $k$ pairs as above let us denote
by $\oC$ the subset
of size $2k$ of $Z_p$ consisting of all elements in the 
union of the $k$ pairs in $C$. 
Note that for 
each fixed indicator $X(P,C)$  as above, the probability it is $1$
is exactly $q^{2k}$, since it is $1$ if and only if all the
$2k$ elements in $\oC$ lie in $A$.
Therefore, by linearity
of expectation and using the fact that 
$k=\Theta(\log p)$ and $q=1/\sqrt p$, 
the expected value $E(X)=\mu$ of $X$ satisfies
$$
E(X)=\mu=(1-o(1))\frac{p^2}{2} \left(\frac{p-1}{2}\right)^k q^{2k}
=(1-o(1))\frac{p^2}{2^{k+1}}.
$$
Our objective is to show that for $k=b \log p$ with the right choice of
the constant $b$, the variance of this random variable is $o(E(X)^2)$.
This will show that $X>0$ whp. Since $E(X)=\mu$ 
tends to infinity as $p$ tends to infinity, it suffices
to prove, using the second moment  method as described, for example,
in \cite{AS}, Chapter 4, that $\Delta=o(\mu^2)$, where $\Delta$ is defined
as follows. It is the sum, over all pairs $X(P,C)$ and 
$X(P',C')$ that are not independent, of the probability
that 
$$
X(P,C)=X(P',C')=1.
$$

We proceed to upper bound $\Delta$. Fixing an indicator random variable
$X(P,C)$ we bound the sum of the conditional probabilities
$\sum Prob[X(P',C')=1~|~X(P,C)=1 ]$ where $X(P,C)$ is fixed 
and $X(P',C')$ ranges over all
the indicators that are not independent of $X(P,C)$.

Note, first, that if $\oC'\cap \oC=\emptyset$ then the above indicator
random variables are independent, hence we may and will assume that
$\oC \cap \oC'$ is nonempty. Put 
$$
C=(\{c_1,c_2\}, \dots ,\{c_{2k-1},c_{2k}\})~~\mbox{and}~~ 
C'=(\{c'_1,c'_2\}, \dots ,\{c'_{2k-1},c'_{2k}\}).
$$ 
Let $\ell$ be the number of pairs
$c'_{2j-1},c'_{2j}$ in which both elements belong to $\oC$. Similarly,
let $m$ be the number of additional indices $j$  so that
$c'_j$ belongs to $\oC$. Therefore $|\oC \cap \oC'|=2\ell+m$. We consider
three possible cases, as follows.
\vspace{0.1cm}

\noindent
{\bf Case 1:}\, $\ell \geq 2$. 

There are ${k \choose \ell}$
ways to choose the indices $j$ for which both elements of the pair
$c'_{2j-1},c'_j$ belong to $\oC$. For each such choice, there are less
than $(2k)^4$ possibilities to choose the specific items of $\oC$
which are equal to the items in the first two pairs above. Given these
choices, and given $C$, we can compute  
the values of the corresponding two terms of the progression
$P'$, and hence get all other terms of the progression. Now each
additional term of the progression $P'$ which is the sum
of two elements of $A$ corresponding to another pair of indices of
$C'$ that lie in $\oC$, has to be a sum of two distinct 
elements $c_j$. There
is at most one way to choose these two $c_j$ (since all sums of pairs
$c_j$ are distinct, by assumption).
Next, there are at most $(2k)^{2m}$  ways to choose the additional
ordered set of $m$ identical elements 
in $\oC$ and in $\oC'$ and decide about the 
bijection between them. It remains to choose the additional
$2k-2\ell-m$ elements of $\oC'$ that have not been determined yet. 
Note that at this point
the progression $P'$ is determined, and
$2\ell+m$ of the values $c'_j$ are determined as well. This also
determines uniquely the values of additional $m$ elements $c'_j$,
since we know the sum of each of the $m$ pairs of elements in which 
one summand is known already. There are still $2k-2\ell-2m$ 
yet undetermined
values $c'_j$ that are partitioned into $k-\ell-m$ pairs, 
where in each of these pairs the sum of elements is known. This gives
$(\frac{p-1}{2})^{k-\ell-m}$ ways to choose the remaining pairs.

Summarizing, we have seen that there are at most
$$
{k \choose \ell} (2k)^4 (2k)^{2m} \left(\frac{p-1}{2}\right)^{k-\ell-m}
$$
indicators $X(P',C')$ corresponding to Case 1
with parameters
$\ell$ and $m$, where, conditioned on $X(P,C)=1$ it is still
possible that $X(P',C')=1$. This last event happens if and only if 
the $2k-2\ell-m$ required elements $c'_j$ that are not
in $\oC$ lie in $A$.
This probability is $\frac{1}{p^{(2k-2\ell-m)/2}}$.
It follows that the total contribution of the sum
$\sum Prob[X(P',C')=1~|~X(P,C)=1 ]$ for a fixed 
$X(P,C)$, that correspond to Case 1 
over all choices of the parameters 
$\ell \geq 2$ and $m \geq 0$ is bounded by
$$
\sum_{\ell \geq 2,m \geq 0} {k \choose \ell} (2k)^4 (2k)^{2m}
\left(\frac{p-1}{2}\right)^{k-\ell -m} \frac{1}{p^{k-\ell-m/2}}
< 2^k (2k)^4 \sum_{m \geq 0} \left[\frac{(2k)^4}{p}\right]^{m/2}
\leq 32 k^4 2^k.
$$
Multiplying by the probability that $X(P,C)=1$ and summing
over all our indicators, we get a bound for the total contribution to
$\Delta$  arising  from pairs that correspond to Case 1.
This total contribution is bounded by
$$
\Delta_1 \leq \mu \cdot 32 k^4 2^k.
$$
\vspace{0.1cm}

\noindent
{\bf Case 2:}\, $\ell = 1$. 

The discussion here is similar, so we only describe it briefly.
There are $k$ possibilities  to choose the index $j$
of the pair $c'_{2j-1},c'_{2j}$ of indices 
contained in $\oC$ and less than $(2k)^2$
ways to choose the corresponding indices in $C$. This
determines one term of the progression $P'$, so there are
less than $p$ ways to choose the whole progression. Next, there
are at most $(2k)^{2m}$ ways to choose the additional
elements of $\oC'$ that lie in $\oC$ and the indices of
them in $C$. Then there are at most $(\frac{p-1}{2})^{k-1-m}$
to select the additional elements of $C'$.
This gives a total of at most
$$
k(2k)^2 p(2k)^{2m} \left(\frac{p-1}{2}\right)^{k-1-m}
$$
relevant indicators $X(P',C')$. The conditional
probability of each of them to be $1$ assuming that
$X(P,C)=1$ is $\frac{1}{p^{(2k-2-m)/2}}$. The total contribution from
these terms to $\Delta$ is therefore at most
$$
\Delta_2 \leq \mu 4k^3 p \cdot 
\sum_{m \ge 0} (2k)^{2m} \frac{1}{p^{m/2}}
= \mu (4k^3) p \sum_{m\ge 0} \left[\frac{(2k)^4}{p}\right]^{m/2}
<\mu \cdot 8k^3 p.
$$
\vspace{0.1cm}

\noindent
{\bf Case 3:}\, $\ell = 0$. 

Following the same reasoning as before the contribution to
$\Delta$ from this case is bounded by
$$
\Delta_3 \leq \mu \cdot \sum_{m \geq 1} (2k)^{2m} 
\left(\frac{p-1}{2}\right)^{k-m} 
\frac{1}{p^{(2k-m)/2}} < \mu \sum_{m \geq 1}
\left[\frac{(2k)^4}{p}\right]^{m/2} 
=\mu \cdot o(1).
$$

Summing $\Delta_1,\Delta_2 $ and $\Delta_3$ we conclude that
$$
\Delta \leq \mu [32 \cdot 2^k k^4+8k^3 p+o(1)].
$$ 
Since $\mu=(1+o(1))\frac{p^2}{2^{k+1}}$ it follows that if, 
say, $k=0.999 \log_2 p$,
then $\Delta=o(\mu^2)$ and then $X>0$ whp. This completes the proof
of the lower bound, and the assertion of Theorem \ref{t11} follows.

\section{Concluding remarks}
\label{s5}

The main covering lemma shows that sumsets of sets with small 
doubling contain dense subsets which have low complexity. 
We believe that this result will have further potential 
applications beyond those discussed in the present paper. In particular:
\begin{itemize}
    \item In follow-up work, leveraging the main covering lemma 
	  and the combinatorial approach to sets with small 
		doubling \cite{CFPY}, we derive sharp asymptotics 
		for the independence number of random Cayley graphs 
		in $\mathbb{Z}_n$ with density $(\log n)^{-2+o(1)}$. 
		This improves and surpasses barriers in earlier 
		works \cite{GM16, CDM, N}.
    \item Combining the probabilistic perspective on sumsets of sets 
	    with small doubling with further inputs, we will address 
		a question of the first author, Balogh, Morris, 
		and Samotij \cite{ABMS} about accurate estimates of 
		the number of sets with small doubling and the 
		typical structure of sets with small doubling. 
\end{itemize}
An optimal dependence on the doubling $K$ in the main covering 
lemma, Theorem \ref{thm:eff-cover}, would be very interesting. 
In particular, the following conjecture, if true, would give 
optimal obstructions that characterize the independence number 
of sparse random Cayley graphs up to logarithmic factors. 

\begin{conj}\label{conj:cover}
    There exist collections of sets $\mathcal{F}_\ell$ such that 
    \[
        \log |\mathcal{F}_\ell| = \tilde{O}(2^{\ell}),
    \]
    and 
    \[
 \min_{F\in \mathcal{F}_\ell} |F| = \tilde{\Omega}(2^{\ell}s),
    \]
    such that for every $|A|=s$ with $|A+A|\le K|A|$, there 
	exists $\ell \le \log_2 K$ and $F\in \mathcal{F}_\ell$ 
	such that $A+A\supseteq F$. 
\end{conj}

We believe that it is also interesting to study further properties 
of the low-complexity subsets of sumsets arising in 
Theorem \ref{thm:eff-cover}. In particular, this would address a  version of Lovett's question appearing in Green's 
list of open problems \cite{Green-op}, which asks if sumsets 
$A+A$ of dense sets contain large subsets which are iterated 
sumsets $B+B+B+B$. 

Regarding arithmetic progressions in random sumsets $A+A$: 
\begin{itemize}
\item
For random subsets $A \subset Z_p$ obtained by picking each element
with probability $q=1/\sqrt p$
it is also possible to prove the upper bound established here
using the first moment method,
where one can show that the main
contribution to the expectation is given by
$k$-progressions 
expressed by sums of $2k$ distinct terms. Indeed, to 
bound the contribution of 
the other possibilities one can consider the connected components 
of the graph
whose edges represent the pairs that sum to the required elements.
This gives an upper bound of $(2+o(1)) \log_2 p$. 
For random sets obtained with larger
probability
$q=C\sqrt p$, where $C>1$, this first moment 
argument does not seem to provide any
nontrivial upper bound, but the argument described here using Talagrand's
Inequality does provide a bound of some $B(C) \log p$
where $B(C)$ is a finite constant for every fixed $C$.
\item
Let $A=A(q) \subset Z_p$ be a random subset of  $Z_p$ obtained by
picking each element of $Z_p$, randomly and independently,
with probability $q=q(p)$. 
Combining the methods here with the Brun Sieve it may be possible
to determine the typical asymptotic behavior of
$ap(A+A)$ up to a constant factor, or even up to a
$(1+o(1))$ factor, for every possible 
$q=q(p)$. For sufficiently small values of $q$ this is essentially done
(using a different method) in \cite{KM}. 
We hope to return to this problem in the future.
\end{itemize}
\vspace{0.2cm}

\noindent
{\bf Acknowledgements}
We thank Yoshi Kohayakawa for helpful comments about 
the content of Section \ref{sec:ap}, and thank Shachar Lovett for
fruitful discussions. Part of this research was performed 
during a visit of the authors, funded by the Clay Institute, in the
program on Extremal and Probabilistic Combinatorics 
at the IAS/Park-City Mathematics Institute in July, 2025. We thank the
Clay Institute and the organizers of the program for their support. Part of this work was completed while the second author was at the Institute for Advanced Study and the SLMath program on Extremal Combinatorics.

\end{document}